\documentclass[11pt]{amsart}
\usepackage{amssymb,amsmath,amsthm}
\usepackage{enumerate, url, verbatim}
\usepackage{color}
\usepackage[all]{xy}

\newtheorem{theorem}{Theorem}
\newtheorem{lemma}[theorem]{Lemma}
\newtheorem{corollary}[theorem]{Corollary}

\newtheorem{proposition}[theorem]{Proposition}

\theoremstyle{remark}
\newtheorem{definition}[theorem]{Definition}

\newtheorem{remark}[theorem]{Remark}

\numberwithin{theorem}{section} \numberwithin{equation}{section}

\newcommand{\ZZ}{\mathbb{Z}}

\newcommand{\al}{\alpha}
\newcommand{\be}{\beta}
\newcommand{\ga}{\gamma}

\newcommand{\mfP}{\mathfrak{P}}

\newcommand{\mfa}{\mathfrak{a}}

\newcommand{\ov}[1]{\overline{#1}}

\newcommand{\isom}{\simeq}

\newcommand{\bmu}{\mbox{\boldmath$\mu$}}

\newcommand{\mfp}{\mathfrak{p}}
\newcommand{\mfb}{\mathfrak{b}}

\newcommand{\mff}{\mathfrak{f}}
\newcommand{\mfc}{\mathfrak{c}}

\newcommand{\ls}{\ell^*}
\newcommand{\F}{\mathbb{F}}
\newcommand{\f}{\mathfrak{f}}
\newcommand{\z}{\zeta}
\newcommand{\Cl}{{\text {\rm Cl}}}

\newcommand{\Q}{\mathbb{Q}}

\newcommand{\Z}{\mathbb{Z}}
\newcommand{\SL}{{\text {\rm SL}}}

\newcommand{\p}{\mathfrak p}
\renewcommand{\a}{\mathfrak a}
\newcommand{\q}{\mathfrak q}
\newcommand{\ze}{\zeta_{\ell}}
\renewcommand{\c}{\mathfrak c}

\newcommand{\N}{{\mathcal N}}

\newcommand{\Gal}{{\text {\rm Gal}}}
\newcommand{\Sym}{{\text {\rm Sym}}}

\newcommand{\Disc}{\textnormal{Disc}}

\renewcommand{\b}{{\mathfrak b}}

\newcommand*{\longhookrightarrow}{\ensuremath{\lhook\joinrel\relbar\joinrel\rightarrow}}

\makeatletter
\let\@@pmod\pmod
\DeclareRobustCommand{\pmod}{\@ifstar\@pmods\@@pmod}
\def\@pmods#1{\mkern4mu({\operator@font mod}\mkern 6mu#1)}
\makeatother

\begin{document}
\title[Identities for field extensions]
{Identities for field extensions generalizing the Ohno--Nakagawa relations}
\date{\today}
\author{Henri Cohen}
\address{Universit\'e de Bordeaux, Institut de Math\'ematiques, U.M.R. 5251 du C.N.R.S,
351 Cours de la Lib\'eration,
33405 TALENCE Cedex, FRANCE}
\email{Henri.Cohen@math.u-bordeaux1.fr}

\author{Simon Rubinstein-Salzedo}
\address{Department of Statistics, Stanford University, 390 Serra Mall, Stanford, CA 94305, USA}
\email{simonr@stanford.edu}

\author{Frank Thorne}
\address{Department of Mathematics, University of South Carolina, 1523 Greene Street, Columbia, SC 29208, USA}
\email{thorne@math.sc.edu}


\begin{abstract}
In previous work, Ohno \cite{O} conjectured, and Nakagawa \cite{N} proved, relations between the counting functions of certain cubic fields. These relations
may be viewed as complements to the Scholz reflection principle, 
and Ohno and Nakagawa deduced them as consequences of `extra functional equations' involving the Shintani zeta
functions associated to the prehomogeneous vector space of binary cubic forms.

In the present paper we generalize their result by proving a similar identity
relating certain degree $\ell$ fields with Galois groups $D_\ell$ and $F_\ell$
respectively, for any odd prime $\ell$, and in particular we give another proof of the Ohno--Nakagawa 
relation without appealing to binary cubic forms.
\end{abstract}

\maketitle

\section{Introduction}
Let $N_3(D)$ denote the number of cubic fields of discriminant $D$. The starting point of this paper is the following theorem
of Nakagawa \cite{N}, which had been previously conjectured by Ohno \cite{O}.
\begin{theorem}\cite{N, O}\label{thm_NO}
Let $D \neq 1, -3$ be a fundamental discriminant.
We have
\begin{equation}
N_3(D^*)+N_3(-27D)=\begin{cases}
N_3(D) & \text{ if $D<0$\;,}\\
3N_3(D)+1 & \text{ if $D>0$\;,}\end{cases}
\end{equation}
where
$D^* = -3D$ if $3 \nmid D$ and $D^* = -D/3$ if $3 \mid D$.
\end{theorem}

Their result is closely related to that which can be derived from 
the classical reflection principle of Scholz \cite{scholz}, which omits the terms
$N_3(-27D)$ and provides for two possibilities for each term on the right.
The significance of $D^*$ is that $\Q(\sqrt{D^*})$ is the {\itshape mirror field} of $\Q(\sqrt{D})$,
the quadratic subfield of $\Q(\sqrt{D}, \z_3)$ distinct from $\Q(\sqrt{D})$
and $\Q(\z_3)$.

Nakagawa deduced his result from a careful study of the arithmetic of binary
cubic forms, which yielded an `extra functional equation' for the associated 
Shintani zeta functions. It appears that such `extra functional equations' 
might be a common feature in the theory of prehomogeneous vector spaces; 
for example, in unpublished work Nakagawa and Ohno \cite{NO} have
conjectured a related formula for the prehomogeneous vector space 
$(\Sym^2 \Z^3 \otimes \Z^2)^*$, which as Bhargava demonstrated in 
\cite{B4, B4count}, may be used to count quartic fields. Nakagawa has made 
substantial headway toward proving this formula, but it appears that there are
still many technical details to be overcome.

In this paper we demonstrate that the Ohno--Nakagawa results can be 
generalized in a different direction, in which cubic fields are replaced
by certain degree $\ell$ fields for any odd prime $\ell$, using a
framework involving class field theory and Kummer theory, and which also
gives another proof of Theorem \ref{thm_NO}.

\medskip
For an odd prime $\ell$, we say that a degree $\ell$ number field is a {\itshape $D_{\ell}$-field} if its Galois closure is dihedral of order $2 \ell$,
and an {\itshape $F_{\ell}$-field} if its Galois closure has Galois group 
$F_{\ell}$, defined by
\begin{equation}\label{def:fl}
F_{\ell} := \langle \sigma, \tau : \sigma^{\ell} = \tau^{\ell - 1} = 1, \tau \sigma \tau^{-1} = \sigma^g \rangle\;,
\end{equation}
for a primitive root $g \pmod \ell$. (Note that different primitive
roots give isomorphic groups, but for our purposes it will be 
important to specify which primitive root is taken.)
For $\ell = 3$ we have $D_3 = F_3 = S_3$, so this distinction is not apparent.

We observe the convention that discriminants always specify the number of pairs of complex embeddings. These will
be indicated by powers of $D$ and $-1$ (e.g., if $D$ is negative, $D^k$ indicates $k$ pairs of complex embeddings and $(-D)^k$ indicates none).
Thus $(-1)^{r_2}|D|^k$
will mean that specific discriminant, with $r_2$ pairs of complex embeddings
(so this is different from, say, $(-1)^{r_2+2}|D|^k$).
Subject to this convention, we write $N_{D_\ell}(D)$ and $N_{F_\ell}(D)$ for the number of $D_\ell$- and $F_\ell$-fields
of discriminant $D$.
Our main theorems, as in Theorem \ref{thm_NO}, 
will relate $N_{D_\ell}(D)$ and $N_{F_\ell}(D')$ for related values of $D, D'$. 

For $\ell > 5$, our methods will not relate {\itshape all} $F_{\ell}$-fields of discriminant $D'$ to $D_{\ell}$-fields; we require an additional
Galois theoretic condition on our $F_\ell$-fields which we now describe. The Galois closure $E'$ of each $F_\ell$-field $E$ that we count
will be a
degree $\ell$ extension of a degree $\ell - 1$ field $K'$, cyclic over $\Q$. 
(See Theorem \ref{thm:char_mf}.) 
In turn, each $K'$ will be a subfield
of the degree $2 (\ell - 1)$ extension $K_z := \Q(\sqrt{D}, \zeta_{\ell})$ (we assume that $D \neq (-1)^{\frac{\ell - 1}{2}} \ell$);
we will call $K'$ the {\itshape mirror field} of
of $K = \Q(\sqrt{D})$.

Choose and fix a primitive root $g \pmod \ell$, and define $\tau$ to be the unique element of $\Gal(\Q(\ze)/\Q)$ with
$\tau(\zeta_\ell) = \zeta_\ell^g$. We write also $\tau$ for the unique lift of this element to $\Gal(K_z/K)$, and for its unique
restriction to an element of $\Gal(K'/\Q)$. (We will have $K \cap K' = \Q$.)

The group $\Gal(K'/\Q)$ acts on $\Gal(E'/K')$ by conjugation, and we require this action to match \eqref{def:fl}
{\itshape for the choices of $\tau$ and $g$ already made}. More precisely, suppose $E'$ is such an extension of $K'$, let $\tau$ denote any lift of the $\tau\in\Gal(K'/\Q)$ from the last paragraph to $\Gal(E'/\Q)$, and let $\sigma\in\Gal(E'/K')\le\Gal(E'/\Q)$ be any element of order $\ell$. Then we require that $\tau\sigma\tau^{-1}=\sigma^g$. (This is independent of the choice of lift of $\tau$ and of $\sigma$.) We write $N^*_{F_\ell}(D)$ for the number of $F_\ell$-fields of discriminant $D$
satisfying this condition.

We will show in Lemma~\ref{discsuffices} that any $F_\ell$ field with the discriminants we count has a mirror field as its $C_{\ell-1}$ subfield. With notation as above 
we must have
$\tau\sigma\tau^{-1}=\sigma^{g'}$ for some primitive root $g'$ modulo $\ell$, so our condition may be stated as requiring that
$g' = g$. Moreover, there are many $F_\ell$ fields whose
discriminants we do not count --- for example, fields of
the form 
$\Q(\sqrt[\ell]{a})$ for $a\in\Q^\times\setminus\Q^{\times\ell}$ and $\ell \geq 5$; the $C_{\ell-1}$ subfield of all these fields is $\Q(\ze)$. Our work raises a variety of questions regarding
the relative frequencies of the fields being counted; we expect that
these questions may be quite difficult to answer, and in any
case we leave them for later investigation.

\smallskip
This brings us to the presentation of our main results:
\begin{theorem}\label{thm:main}
For each negative fundamental discriminant $D \neq -\ell$ we have
\begin{equation}\label{eqn:main_neg}
N_{D_\ell}(D^{\frac{\ell-1}{2}})=\begin{cases}
N^*_{F_\ell}((-1)^0\ell^{\ell-2}|D|^{\frac{\ell-1}{2}})
+
N^*_{F_\ell}((-1)^0\ell^{\ell}|D|^{\frac{\ell-1}{2}})
&\text{ if $\ell\nmid D$}\;,\\
N^*_{F_\ell}((-1)^0\ell^{\frac{\ell - 3}{2}}|D|^{\frac{\ell-1}{2}})
+
N^*_{F_\ell}((-1)^0\ell^{\ell}|D|^{\frac{\ell-1}{2}})
&\text{ if $\ell\mid D$ and $\ell\equiv1\pmod4$}\;,\\
N^*_{F_\ell}((-1)^0\ell^{\frac{\ell - 5}{2}}|D|^{\frac{\ell-1}{2}})
+
N^*_{F_\ell}((-1)^0\ell^{\ell}|D|^{\frac{\ell-1}{2}})
&\text{ if $\ell\mid D$ and $\ell\equiv3\pmod4$}\;.

\end{cases}
\end{equation}
\end{theorem}

For positive discriminants we obtain the 
following close analogue, reflecting the difference between positive
and negative $D$ in the Ohno--Nakagawa relation.

\begin{theorem}\label{thm:main_pos}
For each positive fundamental discriminant $D \neq 1, \ell$ we have
\begin{equation}\label{eqn:main_pos}
\ell N_{D_\ell}(D^{\frac{\ell-1}{2}}) + 1 =\begin{cases}
N^*_{F_\ell}((-1)^{\frac{\ell - 1}{2}}\ell^{\ell-2}D^{\frac{\ell-1}{2}})
+
N^*_{F_\ell}((-1)^{\frac{\ell - 1}{2}}\ell^{\ell}D^{\frac{\ell-1}{2}})
&\text{ if $\ell\nmid D$}\;,\\
N^*_{F_\ell}((-1)^{\frac{\ell - 1}{2}}\ell^{\frac{\ell - 3}{2}}D^{\frac{\ell-1}{2}})
+
N^*_{F_\ell}((-1)^{\frac{\ell - 1}{2}}\ell^{\ell}D^{\frac{\ell-1}{2}})
&\text{ if $\ell\mid D$ and $\ell\equiv1\pmod4$}\;,\\
N^*_{F_\ell}((-1)^{\frac{\ell - 1}{2}}\ell^{\frac{\ell - 5}{2}}D^{\frac{\ell-1}{2}})
+
N^*_{F_\ell}((-1)^{\frac{\ell - 1}{2}}\ell^{\ell}D^{\frac{\ell-1}{2}})
&\text{ if $\ell\mid D$ and $\ell\equiv3\pmod4$}\;.
\end{cases}
\end{equation}\end{theorem}
Nakagawa's Theorem \ref{thm_NO} is the case $\ell = 3$ of these results.

In fact we prove something slightly stronger: The right-hand sides of~(\ref{eqn:main_neg}) and~(\ref{eqn:main_pos}) list two possibilities $\ell^b$ and $\ell^{b'}$ for the power of $\ell$ in the discriminants of $F_\ell$-fields, but they do not rule out other powers of $\ell$ that may occur in $F_\ell$-field discriminants with the desired Galois condition. Our proof (see Proposition \ref{prop:cond_restrict}) shows that in fact there are no $F_\ell$-fields with the given Galois condition and exponents of $\ell$ between 0 and $\frac{3\ell-1}{2}$ other than the ones that appear on the right-hand sides of~(\ref{eqn:main_neg}) and~(\ref{eqn:main_pos}). (Larger exponents do occur, and they do not appear to correspond to $D_\ell$-fields.)

\medskip
A special consideration arises when $\ell \equiv 1 \pmod 4$. Suppose that $d \neq 1$ is a fundamental discriminant not divisible by $\ell$.
Then, $D_\ell$-fields of discriminant $D^{\frac{\ell-1}{2}}$ with $D = d$ and $D = d \ell$ respectively 
correspond to $F_\ell$-fields enumerated on the first and second lines on the right of \eqref{eqn:main_neg} or \eqref{eqn:main_pos}.
It is easily checked that the discriminants and signatures of $F_\ell$-fields enumerated in the first terms on these two lines
(for $D = d$ and $D = d \ell$ respectively) are identical, so that the only
difference between them consists of the condition implied by the star.

It will be proved later that $\Q(\sqrt{d})$ and $\Q(\sqrt{\ell d})$ have the same mirror field when $\ell \equiv 1 \pmod 4$.
However, our definition of $\tau \in \Gal(K'/\Q)$ involved lifting an element of $\Gal(\Q(\ze)/\Q)$ to $\Gal(K_z/K)$ and therefore 
depends on $K$.
Writing $\tau'$ and $\tau''$ for the elements $\tau$ determined when $K = \Q(\sqrt{d})$
and $K = \Q(\sqrt{\ell d})$ respectively, we will see later (in Remark \ref{rem:explain_minus}) that 
the condition $\tau'' \sigma \tau''^{-1} = \sigma^g$ of \eqref{def:fl} is equivalent to
$\tau' \sigma \tau'^{-1} = \sigma^{-g}$. (Note that for a primitive root $g \pmod \ell$ with $\ell \equiv 1 \pmod 4$,
$-g$ is also a primitive root.) 

When $\ell = 5$ there are only two primitive roots, so letting $g$ be either of them we find that all $F_5$-fields
satisfy $\tau' \sigma \tau'^{-1} = \sigma^{g}$ or $\tau' \sigma \tau'^{-1} = \sigma^{-g}$. Therefore, by counting
$D_5$-fields of discriminant with $D = d$ and $D = \ell d$ together we obtain a corresponding count of $F_5$-fields
without any Galois condition:
\begin{corollary}\label{cor:main_cor_5}
If $D$ is a negative fundamental discriminant coprime to 5, we have
\begin{equation}\label{eqn:main_cor_neg}
N_{D_5}\big(D^2\big)
+
N_{D_5}\big(5D^2 \big)
=
N_{F_5}\big((-1)^0 5^3 |D|^2\big)
+ 
N_{F_5}\big((-1)^0 5^5 |D|^2\big)
+
N_{F_5}\big((-1)^0 5^7 |D|^2\big)\;.
\end{equation}
and if $D \neq 1$ is a positive fundamental discriminant coprime to 5, we have 
\begin{equation}\label{eqn:main_cor_pos}
5 \Big( N_{D_5}\big(D^2\big)
+
N_{D_5}\big((5D)^2 \big) \Big) + 2
=
N_{F_5}\big((-1)^2 5^3 D^2\big)
+ 
N_{F_5}\big((-1)^2 5^5 D^2\big)
+
N_{F_5}\big((-1)^2 5^7 D^2\big)\;.
\end{equation}
\end{corollary}

Another (immediate) corollary of our results is that $F_\ell$-fields of certain discriminants must exist.
\begin{corollary}
For each positive fundamental discriminant $D$ coprime to $\ell - 1$,
there exists at least one $F_\ell$-field with discriminant of the form $(-1)^{\frac{\ell - 1}{2}} \ell^a D^{\frac{\ell - 1}{2}}$,
for some $a$ as described above. If $\ell \equiv 1 \pmod 4$, there exist at least two.

\end{corollary}

{\itshape Further directions.}
There are multiple directions in which one might ask for extensions of our results. 
The most obvious is to drop the requirement that $D$ be a fundamental discriminant.
However, as was observed by Nakagawa, no simple relation appears to hold even for $\ell = 3$.
Examining a table
of cubic fields suggests that any result along these lines
would need to account for more subtle information than simply counts of field discriminants.

Similarly, one could attempt to
allow additional factors of $\ell$ in our counts for $D_\ell$-fields. This 
might involve generalizations of the results of Section \ref{sec_gb}, some of which are carried out
in Section 8 of \cite{CTl}, along with further study of the sizes of various groups appearing in these results.

Motivated by Nakagawa's results, one might try to prove a result counting {\itshape ring}
discriminants. In this context, Ohno and Nakagawa {\itshape did} obtain beautiful 
and simple relations among all discriminants, by considering (equivalently):
cubic {\itshape rings} (including reducible and nonmaximal rings); 
binary cubic forms up to $\SL_2(\Z)$-equivalence; or the Shintani zeta functions associated
to this lattice of binary cubic forms. 

The equivalences among these objects do not naturally generalize to $\ell > 3$, and in particular there is no
naturally associated zeta function which is known (to the authors, at least) to have good analytic properties.
Therefore, it seems that the Ohno--Nakagawa relations for cubic rings may be special to the prime $\ell = 3$.
However, it is not out of the question that our work could be extended to an Ohno--Nakagawa
relation counting appropriate subsets of the set of rings of rank $\ell$. In any case work
of Nakagawa \cite{N4} and Kaplan, Marcinek, and Takloo-Bighash \cite{KMTB} (among others) suggests that enumerating
such rings is likely to be quite difficult.



\begin{remark} As F.\ Calegari explained to us, alternative proofs of our results can also be given 
in the language of cohomology and Galois representations,
as a consequence of Poitou-Tate duality \cite{poitou, tate}
and a formula of Greenberg \cite{greenberg} and Wiles \cite{wiles}
(see also Theorem 2.18 of \cite{DDT}).
\end{remark}
\medskip
{\itshape Methods of proof and summary of the paper.} 
The proofs involve the use of class field theory and Kummer theory, along the lines 
developed by the first author
and a variety of collaborators (see, e.g., 
\cite{C_survey, CDO_quartic, CM, CT3, CT4, CTl}) to enumerate fields
with fixed resolvent. Especially relevant is work of the first and third authors \cite{CTl}, giving an explicit 
formula for the Dirichlet series $\sum_K |\Disc(K)|^{-s}$, where the sum is over all $D_{\ell}$-fields $K$ with a {\itshape fixed}
quadratic resolvent. The results of the present paper (or, for $\ell = 3$, of Nakagawa) are required to put this formula
into its most explicit form, as a sum of Euler products indexed by $F_{\ell}$-fields. Our main theorem precisely determines the indexing set of $F_\ell$-fields, and yields the constant term of the main identity of \cite{CTl}.

Our work has an earlier antecedent in the proof of the Scholz reflection principle, as presented for example in Washington's
book \cite{wash}. Let $K, K_z,$ and $K'$ be as described previously.
The technical heart of this paper is the Kummer pairing of Corollary \ref{cor:pairing}, together with its consequence Proposition \ref{prop:size_gb}. 
Our variant of the pairing relates the ray class group $\Cl_{\b}(K_z)/\Cl_{\b}(K_z)^{\ell}$ (for an ideal $\b$ to be described)
with a subgroup of $K_z^{\times} / (K_z^{\times})^{\ell}$ known as an {\itshape arithmetic Selmer group}. Applying a theorem of Hecke will allow us
to conclude, in contrast to the situation in \cite{wash},
that this pairing is perfect.

It is also Galois equivariant, so we can isolate pieces of the ray class group and Selmer group which `come from'
subfields of $K_z$: the Selmer group comes from $K$ (Proposition \ref{prop:KztoK}), and the ray class group comes from $K'$
(Proposition \ref{prop:iso_red}). In Proposition \ref{prop:fl_bij} we will see directly that this ray class group counts $F_{\ell}$-fields.
On the Selmer side our argument will be less direct: 
computations from previous work yield Proposition \ref{prop:size_gb}, relating the size of this Selmer
group to
$|\Cl(K)/\Cl(K)^{\ell}|$. This latter class group counts the $D_\ell$-fields enumerated in our main theorems, as we recall in 
Lemma \ref{lem:dl_disc}.

In Section \ref{sec:prelim} we establish 
a variety of preliminary results on the arithmetic of $D_\ell$ and $F_\ell$-extensions. The most involved result is
Theorem \ref{thm:char_mf}, which guarantees that the Galois closure $E'$ of each $F_{\ell}$-field $E$ we count contains $K'$,
as required for our main theorems to make sense.

In Section \ref{sec_gb} we study the Kummer pairing as described above. We wrap up the proofs in Section \ref{sec:proofs}; 
essentially the only part remaining is to compute the discriminants of the $F_\ell$-fields being counted.
Finally, in Section \ref{sec_numerics} we describe some numerical tests of our results, accompanied by a comment on the {\tt Pari/GP} program (available from the 
third author's website) used to generate them.

\section*{Acknowledgments}
We would like to thank Frank Calegari, J\"urgen Kl\"uners, Hendrik Lenstra, David Roberts, John Voight, and the referees for helpful discussions and suggestions related
to this work.

\section{Preliminaries}\label{sec:prelim}

In this section we introduce some needed machinery and notation, and prove a variety of results about the $D_\ell$- and $F_\ell$-fields counted
by our theorems.
Throughout, $\ell$ is a fixed odd prime.

\subsection{Group theory}
We write $C_r$ for the cyclic group of order $r$ and $D_r$ for the dihedral group of order $2r$. When $r = \ell$ is an odd prime, we write
$F_\ell$ for the Frobenius group defined in \eqref{def:fl}.
The Frobenius group may be realized as the group of affine transformations
$x \mapsto ax + b$ over $\F_{\ell}$ with $a\in\F_\ell^\times$ and $b\in\F_\ell$. The subgroup generated by $\sigma$
(equivalently, the subgroup of translations) is normal, and all nontrivial proper normal subgroups contain $\langle\sigma\rangle$.

The following results are standard and easily checked (granting the basic results of class field theory), and
so we omit their proofs.

\begin{lemma}\label{lem:tower}
Suppose that $K \subset K' \subset K''$ is a tower of field extensions, with $K'/K$, $K''/K'$, and $K''/K$ all Galois, and write $\tau$ and $\sigma$
for elements of $\Gal(K'/K)$ and $\Gal(K''/K')$ respectively. Then:
\begin{enumerate}
\item $\Gal(K'/K)$ acts on $\Gal(K''/K')$ by conjugation; for $\tau \in \Gal(K'/K), \sigma \in \Gal(K''/K')$, the action is defined by
$\tau \sigma \tau^{-1} := \widetilde{\tau} \sigma \widetilde{\tau}^{-1}$ for an arbitrary lift $\widetilde{\tau}$ of $\tau$ to $\Gal(K''/K)$.
\item If further $K''$ corresponds via class field theory to an $\ell$-torsion 
quotient $\Cl_{\mfa}(K')/B$ of a ray class group of $K'$, on which $\tau \in \Gal(K'/K)$ acts by $\tau(x) = x^a$ for some $a \in \mathbb{F}_{\ell}^\times$,
then the conjugation action of $\Gal(K'/K)$ on $\Gal(K''/K')$ is given by $\tau \sigma \tau^{-1} = \sigma^a$.
\end{enumerate}
\end{lemma}

\subsection{Background on conductors}
We recall some basic facts about conductors of extensions of local and global fields, following \cite{serre_lcft}.

\begin{definition} Let $L/K$ be a finite abelian extension of local fields. Let $\mfp$ be the maximal ideal of $\Z_K$. We define the \emph{local conductor} $\mff(L/K)$ to be the least integer $n$ so that \[1+\mfp^n\subseteq N_{L/K}(L^\times)\;.\] \end{definition}

The local conductor thus gives us information about the ramification type of $L/K$. In particular:

\begin{proposition}\label{prop:lc_tower} \begin{enumerate} \item $L/K$ is unramified if $\mff(L/K)=0$, tamely ramified if $\mff(L/K) = 1$, and wildly ramified
if $\mff(L/K) > 1$. 
\item 
If $M/L/K$ is a tower of extensions of local fields with $M/K$ abelian and $L/K$ unramified, then $\mff(M/K)=\mff(M/L)$.
\end{enumerate} \end{proposition}

If $K=\Q_p$, we will sometimes write $\mff(L)$ rather than $\mff(L/\Q_p)$. Also, if $L/K$ is an abelian extension of \emph{global} fields, $\mfp$ a prime of $K$ and $\mfP$ a prime of $L$ above $\mfp$, we will sometimes write $\mff_{\mfp}(L/K)$ for $\mff(L_{\mfP}/K_{\mfp})$, since this does not depend on $\mfP$. Here, $L_{\mfP}$ and $K_{\mfp}$ denote the $\mfP$-adic and $\mfp$-adic completions of $L$ and $K$, respectively.

\begin{definition} Let $L/K$ be a finite abelian extension of global fields, set \[\mff_0(L/K)=\prod_{\mfp} \mfp^{\mff_{\mfp}(L/K)}\;,\] and let $\mff_\infty(L/K)$ denote the set of real places of $K$ ramified in $L$. The \emph{global conductor} of $L/K$ is defined to be the modulus $\mff(L/K)=\mff_0(L/K)\mff_\infty(L/K)$. \end{definition}

\begin{proposition}\label{prop:kw}
If $L/\Q$ is a finite abelian extension, then $\mff_0(L/\Q)$ is the ideal of $\Z$ generated by the least number $n$ so that $L\subseteq\Q(\zeta_n)$. \end{proposition}

\begin{proposition}\label{prop:cd} If $L/K$ be a quadratic extension of global fields, then
$\mff_0(L/K)=\Disc(L/K)$.
\end{proposition}

\subsection{The field diagram}

We fix a primitive $\ell^\text{th}$ root of unity $\ze$ and a primitive root $g\pmod{\ell}$. Let $\ls=(-1)^{\frac{\ell-1}{2}}\ell$, so that $\Q(\sqrt{\ls})$ is the unique quadratic subfield of $\Q(\ze)$.

Let $D$ be a fundamental discriminant, and let $K = \Q(\sqrt{D})$, where we assume that $D \neq \ls$ (although we could presumably handle this case as well).

Write $K_z=K(\ze)$, with $[K_z : \Q] = 2(\ell - 1)$ and 
$\Gamma=\Gal(K_z/\Q) \cong C_2\times(\Z/\ell\Z)^\times$. By Kummer theory, degree $\ell$ 
abelian extensions of $K_z$ are all of the form $K_z(\alpha^{1/\ell})$ for some $\alpha \in K_z$.
Write $\tau$ and $\tau_2$ for the elements of $\Gamma$ fixing $K$ and $\Q(\ze)$ respectively, with $\tau(\ze)=\ze^g$, and $\tau_2$ nontrivial on $K$. 
We also write 
\begin{equation}
T = \{ \tau - g, \tau_2 + 1\},\;T^* = \{ \tau - 1, \tau_2  +1 \} \subseteq \F_{\ell}[\Gamma]\;.
\end{equation}

The {\itshape mirror field $K'$ of $K$} 
is the fixed field of $\tau_2\tau^{\frac{\ell-1}{2}}$; more explicitly,
\begin{equation}\label{eqn:mirror_field}
K'=\Q\big((\ze-\ze^{-1})\sqrt{D}\big) = \Q(\ze + \ze^{-1})\Big( \sqrt{ -D \big(4 - (\ze + \ze^{-1})^2\big)} \Big)\;.
\end{equation}
In particular, $K'$ is a quadratic extension of the maximal totally real subfield of $\Q(\ze)$, it is cyclic of degree $\ell-1$ over $\Q$, with Galois group
generated by the restriction of $\tau$ to $K'$, and its unique quadratic subfield is equal to $\Q(\sqrt{\ell^*})$ if $\ell \equiv 1 \pmod 4$
and to $\Q(\sqrt{D \ell^*})$ if $\ell \equiv 3 \pmod 4$.

We thus have the following diagrams of fields in the $\ell \equiv 1 \pmod 4$ and $\ell \equiv 3 \pmod 4$ cases respectively.

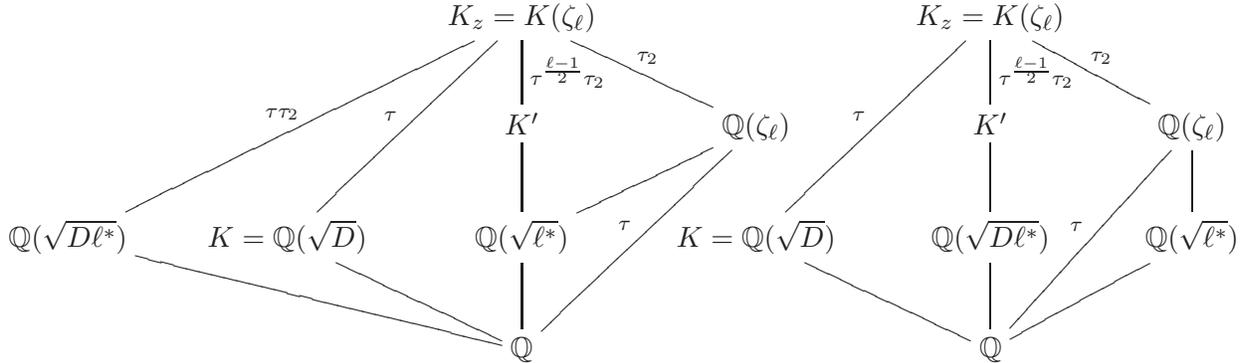
\[\xymatrix{
&& K_z=K(\zeta_\ell)\ar@{-}[ddl]_\tau\ar@{-}[d]^{\tau^{\frac{\ell-1}{2}}\tau_2}\ar@{-}[dr]^{\tau_2}\ar@{-}[ddll]_{\tau \tau_2}
&& K_z=K(\zeta_\ell)\ar@{-}[ddl]_\tau\ar@{-}[d]^{\tau^{\frac{\ell-1}{2}}\tau_2}\ar@{-}[dr]^{\tau_2} 
\\ 
&& K'\ar@{-}[d] & \Q(\zeta_\ell)\ar@{-}[dl]\ar@{-}[ddl]_\tau
& K'\ar@{-}[d] & \Q(\zeta_\ell)\ar@{-}[d]\ar@{-}[ddl]_\tau
\\
\Q(\sqrt{D\ell^\ast})\ar@{-}[drr] & K=\Q(\sqrt{D})\ar@{-}[dr] & \Q(\sqrt{\ell^\ast})\ar@{-}[d] 
& K=\Q(\sqrt{D})\ar@{-}[dr] & \Q(\sqrt{D \ell^\ast})  \ar@{-}[d] & \Q(\sqrt{\ell^\ast})\ar@{-}[dl]
\\ 
&& \Q 
&& \Q}\]

The mirror field of $\Q(\sqrt{D\ell^\ast})$ is fixed by $(\tau\tau_2)^{\frac{\ell-1}{2}}\tau_2$, which is equal to $
\tau^{\frac{\ell-1}{2}}\tau_2$ if $\ell\equiv 1\pmod 4$ and to $\tau^{\frac{\ell-1}{2}}$ if $\ell\equiv 3\pmod 4$. Hence if $\ell\equiv 1\pmod 4$ then the fields $K$ and 
$\Q(\sqrt{D \ell^\ast})$
share the same mirror field, and if $\ell\equiv 3\pmod 4$ they do not. If $\ell = 3$
then $K' = \Q(\sqrt{D \ell^*})$ and $\Q(\ze) = \Q(\sqrt{\ell^*})$, so the second row 
of the diagram should be identified with the third.
\\
\\
{\itshape Notation for splitting types.} We write (as is fairly common) 
that a prime $\mfp$ of a field $K$ has splitting type $(f_1^{e_1} f_2^{e_2} \dots f_g^{e_g})$ in $L/K$ if
$\mfp \Z_L = \mfP_1^{e_1} \mfP_2^{e_2} \cdots \mfP_g^{e_g}$ with $f(\mfP_i | \mfp) = f_i$ for each $i$.

\subsection{Selmer groups of number fields}\label{subsec:selmer}

In Section \ref{sec_gb} our results will be phrased in terms of the $\ell$-Selmer group, which measures the failure of the local-global principle for local $\ell^\text{th}$ powers to be global $\ell^\text{th}$ powers. We recall the relevant terminology here; see also \cite{Co}, \S 5.2.2.

\begin{definition} Let $L$ be a number field. The {\itshape group of $\ell$-virtual units} $V_\ell(L)$ consists of all $u\in L^\times$ for which $u\Z_L=\mfa^\ell$ for some fractional ideal $\mfa$ of $L$, or equivalently all $u\in L^\times$ for which $v_{\mfp}(u)$ is divisible by $\ell$ for all primes $\mfp$ of $L$. The {\itshape $\ell$-Selmer group} is the quotient $S_\ell(L)=V_\ell(L)/L^{\times\ell}$. \end{definition}

If $L=K_z$ then the $\ell$-Selmer group is a finite $\ell$-group, and it fits into a split exact sequence \begin{equation} 1\to\frac{U(K_z)}{U(K_z)^\ell}\to S_\ell(K_z)\to\Cl(K_z)[\ell]\to 1 \label{selexact} \end{equation} of $\F_\ell[\Gamma]$-modules.

\smallskip

In addition we write $\mfb = (1 - \ze)^{\ell} \Z_{K_z}$, and for each $\Gamma$-invariant ideal $\mfc$ of $\Z_{K_z}$ dividing $\mfb$
we write 
\begin{equation}
R_{\mfc} = \Cl_{\mfc}(K_z)/\Cl_{\mfc}(K_z)^{\ell},\; 
G_{\mfc} = R_{\mfc}[T]\;,
\end{equation}
where $T$ has been defined above. 
(For any $\F_{\ell}[\Gamma]$-module
$M$, $M[T]$ denotes the subgroup annihilated by all the elements of $T$.) Because
$\ell$ is totally ramified in $K_z$, any such
$\mfc$ must be of the form $(1 - \ze)^a \Z_{K_z}$ for some integer $a \leq \ell$.

\subsection{The arithmetic of $D_{\ell}$-extensions.}

Our main theorems relate counts of $D_{\ell}$- and $F_{\ell}$-extensions of given discriminant. These fields will be constructed as
subfields of their Galois closures, and our next results (and Proposition \ref{prop:fl_bij}) establish the connection between these two ways of counting
fields.

\begin{lemma}\label{lem:dl_disc}
Let $D$ be a fundamental discriminant. Then
the set of $D_{\ell}$-fields of discriminant $D^{\frac{ \ell - 1}{2}}$ is equal to
the set of degree $\ell$ subfields of unramified cyclic degree $\ell$ extensions $L/\Q(\sqrt{D})$,
and each prime dividing $\ell$ has splitting type $(1^2 1^2 \cdots 1^2 1)$ in each such $D_{\ell}$-field.

In particular, if $k=\Q(\sqrt{D})$, then up to isomorphism there are $\frac{1}{\ell - 1} |\Cl(k)/\Cl(k)^{\ell}|$ of them.
\end{lemma}

Recall that our convention of writing discriminants in the form
$\Disc(F) = (-1)^{r_2(F)} |\Disc(F)|$ specifies the number of complex embeddings of each such field.

\begin{proof}
This can be extracted from Theorem 9.2.6, Proposition 10.1.26,  and Theorem 10.1.28 of \cite{Co}. 
\end{proof}

\begin{remark}\label{rk:is_fd}
Our lemma does not count 
$D_{\ell}$-fields of discriminant
$(4D)^{\frac{\ell - 1}{2}}$ arising from degree $\ell$ extensions of $\Q(\sqrt{D})$ which are ramified at $2$.
An example of such a field is the field generated by a root of $x^3 - x^2 - 3x + 5$ of (non-fundamental) discriminant
$- 2^2 67$.


Related considerations also occur on the $F_{\ell}$ side;
for example, the $F_5$-field generated by a root of $x^5 - 2x^4 + 4x^3 + 12x^2 - 24x + 10$,
of discriminant $(-1)^2 2^4 5^3 53^2$, in which $2$ is totally ramified, is a non-example of a field
counted by our results.
\end{remark}

\subsection{The arithmetic of $F_{\ell}$-extensions}

We now study the arithmetic of $F_{\ell}$-extensions as well as the mirror fields $K'$. The section concludes with
Theorem \ref{thm:char_mf}, which states that  if $E$ is an $F_\ell$-field of appropriate discriminant then
its Galois closure must contain $K'$.

\begin{lemma}\label{lem:tot_ram_mirror_a}
Let $D \neq 1, \pm \ell$ be a fundamental discriminant, and let
$K'$ be the mirror field of $K := \Q(\sqrt{D})$. Then we have
\[
e_{\ell}(K'/\Q) = \begin{cases}
\ell - 1 & \text{\quad if $\ell \nmid D$ or $\ell \equiv 1 \pmod 4$\;,}\\
(\ell - 1)/2 & \text{\quad if $\ell \mid D$ and $\ell \equiv 3 \pmod 4$\;,}\end{cases}
\]
\[
\Disc(K')=\begin{cases}
\ell^{\ell-2}(-D)^{(\ell-1)/2}&\text{\quad when $\ell\nmid D$\;,}\\
\ell^{\ell-2}(-D/\ell)^{(\ell-1)/2}&\text{\quad when $\ell\mid D$ and $\ell\equiv1\pmod4$\;,}\\
\ell^{\ell-3}(-D/\ell)^{(\ell-1)/2}&\text{\quad when $\ell\mid D$ and $\ell\equiv3\pmod4$\;,}\end{cases}\]
and $e_p(K'/\Q) = 2$ for each prime $p \neq \ell$ dividing $D$.
\end{lemma}

\begin{proof} 
Any prime $p \neq \ell$ dividing $D$ is unramified in both $K_z/K$ and $K_z/K'$, so the formula for 
$v_p(\Disc(K'))$ follows by transitivity of the discriminant.

If $\ell \nmid D$, primes above $\ell$ are totally ramified in $\Q(\ze)/\Q$, hence in $K_z/K$, hence {\itshape
not} in $K_z/K'$, hence in $K'/\Q$. If $\ell \mid D$ and $\ell \equiv 1 \pmod 4$, this 
argument with $K$ replaced by
$\Q(\sqrt{\ell^* D})$ yields the same result. 
Finally, if $\ell \mid D$ and $\ell \equiv 3 \pmod 4$, then $\Q(\sqrt{D \ell^*})$ is unramified at $\ell$
and is a subextension of $K'$, so $e_{\ell}(K'/\Q) = (\ell -1)/2$. 
In each of these cases, 
$v_\ell(\Disc(K'))$ is uniquely determined by \eqref{discval} below.

The power of $-1$ in $\Disc(K')$ follows from the formula 
$K'=\Q((\zeta_\ell-\zeta_\ell^{-1})\sqrt{D})$; since $K'$ is Galois, it is either totally real or totally complex.
\end{proof}








\begin{lemma} \label{discsuffices} Suppose that $F/\Q$ is a $C_{\ell-1}$-field, with $|\Disc(F)|$ equal to
$|D|^{\frac{\ell - 1}{2}}$ times some (positive or negative) power of $\ell$ for a fundamental discriminant $D$.
Then $F$ is equal to the mirror field of $\Q(\sqrt{D})$ or $\Q(\sqrt{\ell^\ast D})$, with discriminant
given by Lemma \ref{lem:tot_ram_mirror_a}.
\end{lemma}
In other words, if $F$ has the same discriminant and signature as a mirror field $K'$, then 
$F\cong K'$. If local exceptions are allowed at $\ell$ and infinity, then $F$ must be one
of the fields $K'$ enumerated in Lemma \ref{lem:tot_ram_mirror_a}, and knowing the discriminant and signature suffices 
to determine which.

\begin{proof} First of all, we claim that $e_p(F/\Q)$ is uniquely determined by $\Disc(F)$ for each prime $p$. 
If $p \neq 2$, then $p$ is not wildly ramified in $F$,
and $e_p(F/\Q)$ may be determined from the formula
\begin{equation}
 v_p(\Disc(F)) = (\ell - 1)\Big(1 - \frac{1}{e_p(F/\Q)}\Big). \label{discval} 
 \end{equation} 

If $p = 2$ is ramified in $F$, then $v_2(\Disc(F))$ equals
either $\ell - 1$ or $3(\ell - 1)/2$ and the ramification is wild.
There is a unique intermediate field $\Q\subseteq F' \subseteq F$ with $[F:F']=2$
containing the inertia field. We claim that
$v_2(\Disc(F')) = 0$: if not, by transitivity of the discriminant $v_2(\Disc(F')) = (\ell - 1)/4$, which would imply
that $2$ is ramified in $F'$ with $e_2(F'/\Q) = 2$ by the analogue of \eqref{discval}, which is absurd as $2 \mid [F' : \Q]$.
Therefore $v_2(\Disc(F')) = 0$ and $e_2(F/\Q) = 2$.

We also note that each other prime $p \not \in \{ 2, \ell \}$ which ramifies in $F$ satisfies $e_p(F/\Q) = 2$ and $p$ is unramified in $F'$.

The inertia groups generate $\Gal(F/\Q)$ because they generate a subgroup of $\Gal(F/\Q)$ whose fixed field is everywhere unramified.   
If $\ell\equiv 1\pmod 4$ the inertia group at $\ell$ must therefore be all of $C_{\ell-1}$.
If $\ell\equiv 3\pmod 4$ the inertia group could be the full Galois group or its index $2$ subgroup, and these two cases may be
distinguished by $v_\ell(\Disc(F))$.

\medskip
Write $D_1 = \ell |D|$ if $\ell \nmid D$ and $D_1 = |D|$ if $\ell \mid D$.
By Proposition \ref{prop:kw} we have
$F \subseteq \Q(\zeta_{D_1})$, as we see by computing local conductors:
each prime $p \neq \ell$ is unramified in $F'$, so that by Propositions
\ref{prop:lc_tower} and \ref{prop:cd} and transitivity of the discriminant
we have $\f_p(F)=v_{\mfp}(\Disc(F_{\mfP}/F'_{\mfp})) = v_p(D)$, where $\mfp$ and $\mfP$
are primes of $F'$ and $F$ above $p$ and $\mfp$ respectively.
Moreover, the prime $\ell$ is tamely
ramified in $F$ so that $\f_\ell(F) = 1$ by Proposition \ref{prop:lc_tower}.

Write $\Gal(\Q(\zeta_{D_1})/\Q)$ as $\prod_{p^{a_p}\mid \mid D_1} (\Z / p^{a_p})^{\times}$
and
$\Gal(\Q(\zeta_{D_1})/F)=A\subset\Gal(\Q(\zeta_{D_1})/\Q)$.
For each $p$, $A\cap (\Z/p^{a_p})^{\times}$ is the inertia group of primes above $p$
in $\Q(\zeta_{D_1}/F)$, so that multiplicativity of ramification degrees implies that
$[(\Z / p^{a_p})^{\times} : A\cap (\Z/p^{a_p})^{\times}]= e_p(F/\Q)$.

Write $B_p := (\Z / p^{a_p})^{\times}$ and $B'_p := A \cap B_p$ for each $p$. For $p \not \in \{ 2, \ell \}$
$B'_p$ is the unique index $2$ subgroup of $B_p$, and $B'_\ell$ is either trivial or the unique order
$2$ subgroup of $B_\ell$, as determined above by $v_\ell(\Disc(F))$. $B'_2$ is of index $2$ in $B_2$;
if $4 \mid \mid D$, then $B'_2$ is uniquely determined, and if $8 \mid \mid D$ there are two possibilities
for $B'_2$.
We claim that this information uniquely determines $A$, except in the $8 \mid \mid D$ case where both possibilities
can occur. Since the mirror fields of Lemma \ref{lem:tot_ram_mirror_a} satisfy all the same properties, this claim
establishes the lemma. 

The claim is easily checked: There is a unique subgroup $B \subseteq \prod_{p \neq \ell} B_p$ of index $2$ containing
$\prod_{p \neq \ell} B'_p$; it consists of vectors $(b_p)_{p \neq \ell}$ for which $b_p \not \in B'_p$ for an even number of $p$.
Moreover, $B_\ell$ contains a unique element $b_\ell$ of order $2$.
If $e_{\ell} = \ell - 1$, then $A$ must consist of
$\{1\} \times B$ and $\{b_\ell\} \times \big( \prod_{p \neq \ell} B_p - B \big)$.
If $e_{\ell} = \frac{\ell - 1}{2}$, then $A = \{1, b_\ell \} \times B$; to see that no other $\ell$-component
is possible, we use the fact that $\ell \equiv 3 \pmod 4$ to see that $B_\ell$ contains no elements of order $4$.
 \end{proof}

At this point we highlight the \textit{Brauer relation} (see \cite[Theorems 73 and 75]{FT}): If $E/\Q$ is a degree $\ell$ extension with Galois closure $E'$ with $\Gal(E'/\Q)\cong F_\ell$, and $F$ is the $C_{\ell-1}$ subextension of $E'$, then 
\begin{equation}\label{eqn:brauer_1}
\zeta(s)^{\ell-1}\zeta_{E'}(s)=\zeta_{E}(s)^{\ell-1}\zeta_F(s)\;,
\end{equation} 
which implies that \begin{equation}
\Disc(E') = \Disc(E)^{\ell - 1} \Disc(F)\;. \label{brauerreln}
\end{equation} (This relation also holds true for the infinite place.) This follows from a computation involving the characters of $F_{\ell}$.

This relation also implies that $\Disc(E)=\Disc(F)\N(\f(E'/F))$, where
$\f(E'/F)$ is the conductor of the abelian extension $E'/F$.

\smallskip
We can now conclude that, given suitable conditions on $\Disc(E)$, $F$ must be a mirror field. Later we will apply this
to count these $F_\ell$-fields
using class field theory.

\begin{theorem}\label{thm:char_mf} 
Suppose that $E/\Q$ is an $F_\ell$-field with 
$\Disc(E)$ equal to 
$(-D)^{\frac{\ell-1}{2}}$ times an arbitrary power of $\ell$ for a fundamental discriminant $D$.
Let $E'$ be the Galois closure of $E$, 
and let $F/\Q$ be the unique subextension of degree $\ell-1$. 

Then $E'/F$ is unramified away from the primes dividing $\ell$, and
$F$ is equal to the mirror field $K'$ of $\Q(\sqrt{D})$. \end{theorem}

\begin{proof} For the first claim, it suffices to prove that no prime $p \neq \ell$ can totally ramify in $E/\Q$.
This is immediate for primes $p \not \in \{ 2, \ell \}$, as $v_p(E) < \ell - 1$.
However, the case $p = 2$ is more subtle: Remark \ref{rk:is_fd} illustrates that 
it cannot be treated by purely local considerations.

So suppose to the contrary that $2$ is totally ramified in $E$, so that
$4 \mid \mid D$.
We first claim that $2$ is unramified in $E/E'$ and therefore (because $\ell$ and $\ell - 1$ are coprime) also in $F/\Q$. 
To see this, we work locally. Any totally and tamely ramified extension of $\Q_2$ is of the form $\Q_2(\alpha)$, where $\alpha$ is a root of $x^e-\pi$, where $\pi$ is a uniformizer of $\Q_2$. (See \cite{Lang94}, Proposition 12 in II, \S 5.) Such extensions do not ramify further when we pass to the Galois closure. 

For every other prime $p\neq 2,\ell$ dividing $D$, primes above $p$ are unramified in $E'/F$, so that
\eqref{eqn:brauer_1} and \eqref{discval} imply that $e_p(F/\Q)=2$. 

Therefore, $|\Disc(F)|$ equals $(|D|/4)^{\frac{\ell - 1}{2}}$ times a power of $\ell$, so that $\Disc(F)$ is determined
by Lemma \ref{discsuffices}. In particular, since $-D/4$ is a fundamental discriminant and $D/4$ is not, 
$F$ is totally real if $-(-D/4) = D/4$ is positive, and totally imaginary if $D$ is negative. However, 
the condition for $\Disc(E)$ implies that $E'$, and therefore $F$, is totally real if and only if $-D$ is positive.
We therefore have a contradiction.

\smallskip
Now we conclude from \eqref{brauerreln} that  $\Disc(E) = \ell^c \Disc(F)$ for some $c \geq 0$, so that $F$ satisfies the conditions
of Lemma \ref{discsuffices}. This implies the second claim; when $\ell \equiv 3 \pmod 4$, the possibility that $K'$ is the mirror field
of $\Q(\sqrt{\ell^* D})$ is ruled out because the signature of $E$ determines that of $F$.
\end{proof}

\section{The Kummer pairing and $F_\ell$-fields}\label{sec_gb}

In this section we introduce the Kummer pairing and use it to obtain two different expressions
for the size of the group $G_{\b}$ (introduced at the end of Section \ref{subsec:selmer}), each of which corresponds to one of the field counts in the main theorems.
Ideas for this section were contributed by Hendrik Lenstra, and we thank him for his help.

\smallskip

We begin with the following consequence of a classical result of Hecke.

\begin{proposition} Suppose that $N_z=K_z(\sqrt[\ell]{\alpha})$. Then 
we have $\f(N_z/K_z)\mid\b$ if and only 
if $\alpha$ is an $\ell$-virtual unit.
\end{proposition}

\begin{proof} See Theorem 10.2.9 of \cite{Co}.
\end{proof}

\begin{corollary}\label{cor:pairing} Let $\bmu_\ell$ denote the group of $\ell^\text{th}$ roots of 
unity. There exists a perfect, $\Gamma$-equivariant pairing of 
$\F_\ell[\Gamma]$-modules
$$R_{\b}\times S_{\ell}(K_z)\to\bmu_\ell\;.$$
\end{corollary}

\begin{proof} This is simply the Kummer pairing: let $M/K_z$ be the
abelian $\ell$-extension corresponding by class field theory to
$R_{\b}$, which is the compositum of all cyclic degree $\ell$ extensions of $K_z$ with conductors
dividing $\b$. If $\ov{\a}\in R_{\b}$, we denote as usual by 
$\sigma_{\a}\in\Gal(M/K_z)$ the image of $\a$ under the Artin map. Thus, by the above 
proposition, if $\ov{\al}\in S_{\ell}(K_z)$ and $\al$ is virtual unit representing $\ov{\al}$, we have 
$K_z(\root\ell\of\al)\subset M$, and we define the pairing by
$$(\ov{\a},\ov{\al})\mapsto \sigma_{\a}(\root\ell\of\al)/\root\ell\of\al\in\bmu_\ell\;,$$
which does not depend on any choice of representatives. It is classical and immediate that this pairing is perfect
and $\Gamma$-equivariant, e.g., that
$\langle \tau_1(\ov{\a}),\tau_1(\ov{\al}) \rangle =\tau_1(\langle \ov{\a},\ov{\al}\rangle)$ for any $\tau_1 \in \Gamma$.
\end{proof}

\begin{corollary}\label{cor:gb_pairing} We have a perfect pairing
$$G_{\b}\times S_{\ell}(K_z)[T^*]\to\bmu_\ell\;.$$
In particular, we have
$$|G_{\b}| = |S_{\ell}(K_z)[T^*]|\;.$$
\end{corollary}

\begin{proof} Applying the $\Gamma$-equivariance of the pairing
of the preceding corollary, and recalling that $\tau(\z_\ell) = \z_\ell^g$, for any $j$ we obtain 
a perfect pairing 
$$R_{\b}[\tau-g^j]\times S_{\ell}(K_z)[\tau-g^{1-j}]\to\bmu_\ell\;.$$

Taking $j=1$ yields a perfect pairing between 
$R_{\b}[\tau-g]$ and $S_{\ell}(K_z)[\tau-1]$, and similarly, since $\tau_2$ leaves $\z_\ell$ fixed, we 
obtain a perfect pairing between $G_{\b}=R_{\b}[\tau-g,\tau_2+1]$ and
$S_{\ell}(K_z)[\tau-1,\tau_2+1]$.
\end{proof}

\begin{proposition} \label{prop:KztoK}
We have $S_{\ell}(K_z)[T^*]\isom S_{\ell}(K)$.
\end{proposition}

\begin{proof} We have an evident injection
\[
S_{\ell}(K) \longhookrightarrow S_{\ell}(K_z)[\tau-1]\;,
\]
which is also surjective: if $\al\in K_z$ satisfies 
$\tau(\al)/\al=\ga^\ell$ for some $\ga, x \in K_z$,
then $\N_{K_z/K}(\ga)^\ell= \N_{K_z/K}(\ga)=1$ (since $\z_\ell\notin K$). By 
Hilbert 90 applied to $K_z/K$ there exists $\be\in K_z$ with
$\ga=\be/\tau(\be)$, hence $\tau(\al\be^\ell)/(\al\be^\ell)=1$, so
$a=\al\be^\ell$ is a virtual unit of $K_z$, and also of $K$ because
$([K_z:K], \ell) = 1$.

Therefore $S_{\ell}(K_z)[T^*]=S_{\ell}(K_z)[\tau-1,\tau_2+1] \isom S_{\ell}(K)[\tau_2+1]$.
On the other hand we have trivially
$$S_{\ell}(K)=S_{\ell}(K)[\tau_2+1]\oplus S_{\ell}(K)[\tau_2-1]\;,$$
and we claim that $S_\ell(K)[\tau_2-1]$ is trivial: if $\al\in K$
satisfies $\tau_2(\al)=\al\ga^\ell$ for some $\ga\in K$, then applying $\tau_2$ again we
deduce that $(\ga\tau_2(\ga))^\ell=1$ and thus $\ga\tau_2(\ga)=1$, so that by 
a trivial case of Hilbert 90, $\ga=\tau_2(\be)/\be$ for some 
$\be\in K$, hence $\tau_2(\al/\be^\ell)=\al/\be^\ell$. Thus $\al/\be^\ell$ is 
a virtual unit of $\Q$ equivalent to $\al$, and since $S_\ell(\Q)$ is trivial
this proves our claim and hence the proposition.
\end{proof}

We therefore have the equality $|G_{\b}|=|S_{\ell}(K)|$, which we use to obtain the following: 

\begin{proposition}\label{prop:size_gb}
We have
\begin{equation}
|G_{\b}| = 
\begin{cases}
|\Cl(K)/\Cl(K)^{\ell}| & \text{ if $D < 0$}\;, \\
\ell |\Cl(K)/\Cl(K)^{\ell}| & \text{ if $D > 0$}\;.
\end{cases}
\end{equation}
\end{proposition}

\begin{proof} 
By the exact sequence (\ref{selexact}) and Proposition 2.12 of \cite{CoDiOl3}, the proofs of which adapt to $K$
without change, and since $\dim_{\F_\ell}(U(K)/U(K)^\ell)=1-r_2(D)$, 
where (as usual) $r_2 = 1$ if $D < 0$ and $r_2 = 0$ if $D > 0$,
we obtain
$$|S_{\ell}(K)|=\ell^{1-r_2(D)}|\Cl(K)/\Cl(K)^\ell|\;,$$
yielding the proposition.
\end{proof}

Note that the last statement generalizes Proposition 7.7 of \cite{CM}.

\smallskip

By Lemma \ref{lem:dl_disc} it thus follows that $D_{\ell}$-fields can be counted in terms of $G_{\mfb}$. 
We now show that the same is true of $F_\ell$-fields. We begin by showing that $G_{\b}$ can be `descended' to
$K'$,  generalizing Proposition 3.4 of \cite{CT3}:

\begin{proposition}\label{prop:iso_red}
Let $\mfc = (1 - \ze)^a \Z_{K_z}$ be any $\Gamma$-invariant ideal dividing $\mfb$.
\begin{enumerate}[(1)]
\item
There is an isomorphism
\[
\frac{ \Cl_{\mfc}(K_z)}{\Cl_{\mfc}(K_z)^{\ell}}[T] \rightarrow \frac{ \Cl_{\mfc'}(K')}{ \Cl_{\mfc'}(K')^{\ell} }[\tau - g]\;,
\]
where $K'$ is the mirror field of $K = \Q(\sqrt{D})$
and $\mfc' = \mfc \cap K'$. 
\item
We have
\begin{enumerate}\item $\c' = \p^a$ if either $\ell$ is unramified
in $K$ or $\ell\equiv1\pmod4$, where $\p$ is the unique prime of $K'$ above
$\ell$;
\item $\c' = \p^{\lceil \frac{a}{2} \rceil}$ if $\ell$ is ramified in $K$ and
$\ell\equiv3\pmod4$, where $\q=\p$ or $\q=\p\p'$ depending on whether
there is a unique prime $\p$ or two distinct primes $\p$ and $\p'$
of $K'$ above $\ell.$
\end{enumerate}
\end{enumerate}
\end{proposition}

\begin{proof}
Since $\tau_2$ and $\tau^{(\ell-1)/2}$ each act by $-1$ on $G_{\mfc}$, $\tau^{(\ell-1)/2}\tau_2$ acts trivially.
Writing $e = \frac{1+\tau_2\tau^{(\ell-1)/2}}{2}$, decomposing $1 = e + (1 - e) = \frac{1+\tau_2\tau^{(\ell-1)/2}}{2} + \frac{1-\tau_2\tau^{(\ell-1)/2}}{2}$
in $\F_\ell[\Gamma]$, and noting that 
$G_{\mfc}$ is annihilated by $1-e$, we see that
the elements of $G_{\mfc}$ are exactly those elements of $\frac{\Cl_{\mfc}(K_z)}{\Cl_{\mfc}(K_z)^\ell}$ that can be represented by an ideal of the form $\mfa\tau_2\tau^{(\ell-1)/2}(\mfa)$, which we check is of the form $\mfa'\Z_{K_z}$ for some ideal $\mfa'$ of $K'$. 

As we check, we obtain a well-defined, injective map
$G_{\mfc}\to\frac{\Cl_{\mfc'}(K')}{\Cl_{\mfc'}(K')^\ell}[ \tau - g]$. To see that it is surjective, observe that
any class in $\frac{\Cl_{\mfc'}(K')}{\Cl_{\mfc'}(K')^\ell}[\tau-g]$ is represented by $I \sim I^{1 + \ell}$ for some ideal $I$ of $\Z_{K'}$, 
and with $\mfa=I^{(1+\ell)/2}$ we have
$I^{1+\ell}\ZZ_{K_z} = \mfa\tau_2\tau^{(\ell-1)/2}(\mfa)$.

For (2a), recall that $\ell$ is totally ramified in $K'$ by Lemma \ref{lem:tot_ram_mirror_a}, so that we must show
that $\mfc\cap K'=\mfp^a$. As $\ell$ is unramified in $\Q(\sqrt{D})$, we have $e_{\ell}(K_z/\Q) = \ell - 1$, and if
$\mfP$ is a prime of $K_z$ above $\mfp$ then $v_{\mfp}(x) = v_{\mfP}(x)$ for any $x \in K'$, hence the result.

For (2b), Lemma \ref{lem:tot_ram_mirror_a} implies that $\ell$ has ramification index $\frac{\ell - 1}{2}$ in $K'$,
and hence that each prime of $K'$ above $\ell$ has ramification index $2$ in $K_z/K'$.
That is, $2 v_{\mfp}(x) = v_{\mfP}(x)$, and the result
follows.
\end{proof}

We can now obtain the desired bijection for $F_\ell$-fields, 
adapting Proposition 4.1 in \cite{CT3}. 

\begin{proposition}\label{prop:fl_bij} 
For each $\Gamma$-invariant ideal $\mfc \mid \mfb$, there exists a bijection between the following two sets:
\begin{itemize}
\item
Subgroups of index $\ell$ of $G_{\mfc} = \frac{ \Cl_{\mfc}(K_z)}{\Cl_{\mfc}(K_z)^{\ell}}[T]$. 
\item
Degree $\ell$ extensions $E/\Q$ (up to isomorphism), whose Galois closure $E'$ has Galois group $F_{\ell}$
and contains $K'$, with the conductor $\mff(E'/K')$ dividing $\mfc' = \mfc \cap K'$, such that
$\tau \sigma \tau^{-1} = \sigma^g$ for any generator $\sigma$ of $\Gal(E'/K')$. 
\end{itemize}
\end{proposition}

\begin{remark}
Recall that the element $\tau \sigma \tau^{-1} \in \Gal(E'/K')$ is well defined by Lemma \ref{lem:tower}. Also, note that $\mff(E'/K')$ is $\Gamma$-invariant, because $E'$ is fixed by $\tau_2$, courtesy of Proposition~\ref{prop:iso_red}.
\end{remark}

\begin{proof}
%
By Proposition \ref{prop:iso_red}, it suffices to exhibit a bijection between the set of field extensions as above,
and
subgroups of index $\ell$ of $G'_{\c'} := \frac{ \Cl_{\mfc'}(K')}{ \Cl_{\mfc'}(K')^{\ell} }[\tau - g],$ where $\mfc' = \mfc \cap K$.

Given such a subgroup, we produce a degree-$\ell$ extension of the desired type. 
Write $A':=\Cl_{\mfc'}(K')/\Cl_{\mfc'}(K')^\ell$, and decomposing $A'$ into eigenspaces for 
the action of $\tau$ (as we can, because the order of $\tau$ is coprime to $\ell$)
write $A'\cong G'_{\c'} \times A''$ where $A''$ is the sum of the other eigenspaces.

Subgroups $B \subseteq G'_{\c'}$ of index $\ell$ are in bijection with subgroups
$B'=B\times A'' \subseteq A'$ of index $\ell$ containing $A''$. For each $B'$, 
class field theory gives a unique extension $E'/K'$, cyclic of degree $\ell$, of conductor
dividing $\mfc'$,
for which the Artin map induces an isomorphism $G_{\mfc'}'/B'\cong\Gal(E'/K')$. Furthermore, $E'$ is Galois over $\Q$ because $G_{\mfc'}$ and $B'$ are $\tau$-stable.
Each $B$ yields a distinct $E'$, and as the action of $\Gal(K'/\Q)$ on the class group matches the conjugation
action of $\Gal(K'/\Q)$ on $\Gal(E'/K')$ we have $\Gal(E'/\Q) \isom F_{\ell}$ with presentation as in the second bullet
point. The extension $E$ may be taken to be any of the isomorphic
degree $\ell$ subextensions of $E'$.

Finally we note that all the steps are reversible, establishing the desired bijection.
\end{proof}

\begin{remark}\label{rem:explain_minus}
We now justify the remark made after the statement of our main results concerning the notation * and the primitive roots $\pm g$.

Suppose that $\ell \equiv 	1 \pmod 4$, that $\ell \nmid D$, and that $\tau$ is a generator of $\Gal(K_z/K)$, so that
$\tau \tau_2$ is a generator of $\Gal(K_z/\Q(\sqrt{D \ell^*}))$. Then both $K$ and 
$\Q(\sqrt{D \ell^*}) = \Q(\sqrt{D \ell})$ have the same mirror field.

Replacing $K$ with $\Q(\sqrt{D \ell})$ is equivalent to replacing $\tau$ with $\tau \tau_2$
and thus $T = \{ \tau - g, \tau_2 + 1 \}$ with $\{ \tau \tau_2 - g, \tau_2 + 1 \}$, or equivalently, $\{\tau + g, \tau_2 + 1 \}$.
Thus, if we study $D_\ell$-extensions with resolvent $\Q(\sqrt{D \ell})$, where $\tau$ is still regarded
as a generator of $\Gal(K_z/\Q(\sqrt{D}))$, we obtain the same results with $g$ replaced with $-g$.
In particular, in the previous lemma we obtain field extensions $E$ with
$\tau \sigma \tau^{-1} = \sigma^{-g}$.

\end{remark}

We now show that the set of conductors $\mff(E'/K')$ that can occur in Proposition
\ref{prop:fl_bij} is quite limited.

\begin{proposition}\label{prop:cond_restrict}
The conductors $\mff(E'/K')$ of fields counted in Proposition \ref{prop:fl_bij} are restricted to the following values:
\begin{itemize}
\item If $\ell \nmid D$, $v_\ell(\mff(E'/K')) \in \{ 0, 2 \}$.
\item If $\ell \mid D$ and $\ell \equiv 1 \pmod 4$, $v_\ell(\mff(E'/K')) \in \{ 0, \frac{\ell + 3}{2} \}$.
\item If $\ell \mid D$ and $\ell \equiv 3 \pmod 4$, $v_\ell(\mff(E'/K')) \in \{ 0, 2, \frac{\ell + 5}{2} \}$.
\end{itemize}
\end{proposition}
\begin{proof}
We work with the extensions $E''/K_z$ which correspond to the extensions $E'/K'$ by Proposition \ref{prop:iso_red}.
Unraveling the definition of $G_{\c}$, we see that the conductor of such an extension can be equal to
$(1 - \ze)^a \Z_{K_z}$ if and only if
\[
\frac{1 + P^a}{1 + P^{a + 1}} [T] \neq 0,
\]
where $P = (1 - \ze) \Z_{K_z}$ if this ideal is prime, and $P$ is one of the two primes dividing $(1 - \ze) \Z_{K_z}$ otherwise.
The case $a = 0$ is not excluded in any case listed above; so assuming that $a \geq 1$, we use
the inverse Artin-Hasse logarithm and exponential maps, in exactly the same way as on
p.\ 177 of \cite{CoDiOl3}, to conclude that
\begin{equation}\label{eq:cdo_setup}
\frac{P^a}{P^{a + 1}} [T] \neq 0.
\end{equation}
Necessary conditions for \eqref{eq:cdo_setup} were given in Theorem 1.2 of the first author, Diaz y Diaz, and
Olivier's study \cite{CoDiOl1} of cyclotomic fields. In all cases $P$ and $K_z$ have the same meaning here and
in \cite{CoDiOl1}.
\begin{itemize}
\item
If $\ell \nmid D$, then let $K$ have the same meaning as here, and consider the $\tau - g$ eigenspace 
with $e(\p) = 1$. Then Theorem 1.2 implies
that $a \equiv 2 \pmod{\ell - 1}$.
\item
If $\ell \mid D$ and $D \equiv 1 \pmod 4$, let $K$ of \cite{CoDiOl1} be
$\Q(\sqrt{D \ell})$, so that the $T$-eigenspace lies within the $\tau - g^{(\ell + 1)/2}$ eigenspace. Then
Theorem 1.2 implies that
$a \equiv \frac{\ell + 3}{2} \pmod{\ell - 1}$.
\item
If $\ell \mid D$ and $D \equiv 3 \pmod 4$, then again $K$ has the same meaning in \cite{CoDiOl1} as here;
now $e(\p) = 2$, so that $a \equiv 2 \pmod{\frac{\ell - 1}{2}}$.
\end{itemize}
So, given that $a \leq \ell - 1$, we obtain respectively in these three cases for $\mff(E''/K_z)$ that
$a \in \{0, 2 \}$, $a \in \{0, \frac{\ell + 3}{2} \}$, and $a \in \{ 0, 2, \frac{\ell + 3}{2} \}$. 
By Proposition \ref{prop:iso_red} the corresponding values for $\mff(E''/K_z)$ are 
$a$, $a$, and $2 \lceil \frac{a}{2} \rceil$, so $a \in \{0, 2\}$, $a \in \{0, \frac{\ell + 3}{2} \}$, and
$a \in \{0, 2, \frac{\ell + 5}{2} \}$ respectively.
\end{proof}

\section{Proofs of the Main Results}\label{sec:proofs}

\begin{proof}[Proof of Theorem \ref{thm:main}]
Let $K = \Q(\sqrt{D})$ with $D< 0$. The key to the proof is the identity
$|G_{\mfb}| = |\Cl(K)/\Cl(K)^\ell|$
of Proposition \ref{prop:size_gb}.
By Lemma \ref{lem:dl_disc},
$\frac{1}{\ell-1}(|G_{\mfb}| - 1)$ equals 
the number of 
$D_{\ell}$ extensions with discriminant $(-1)^\frac{\ell-1}{2} D^{\frac{\ell-1}{2}}$. 
Simultaneously, Propositions \ref{prop:iso_red} and \ref{prop:fl_bij} imply that $\frac{1}{\ell-1}(|G_{\mfb}| - 1)$ is the number
of $F_{\ell}$ extensions whose Galois closure $E'$ contains the mirror field $K'$, with $\mff(E'/K') \mid \mfb \cap K$, and with 
$\tau \sigma \tau^{-1} = \sigma^g$ as described there. 
Theorem \ref{thm:char_mf} implies that the Galois closure of each $F_{\ell}$-field described in the theorem must contain $K'$,
so that it remains only to prove
that the condition $\mff(E'/K') \mid \mfb \cap K$ coincides with the discriminant conditions on the $F_{\ell}$-fields
counted in the theorem.

First assume that $\ell \equiv 1 \pmod 4$ or $\ell \nmid D$ (or both). Then
Lemma \ref{lem:tot_ram_mirror_a} implies that
$
\Disc(K') = \ell^{\ell-2} (-D)^\frac{\ell-1}{2}
$
or $\Disc(K')=\ell^{\ell-2}(-D/\ell)^{\frac{\ell-1}{2}}$, if $\ell\nmid D$ or $\ell\mid D$, respectively.
Thus,
we have
\begin{equation}
v_\ell(\Disc(E')) =  \ell(\ell-2) + (\ell-1) \mff_{\mfp}(E'/K')\;,
\end{equation}
where $\mfp$ is the unique (totally ramified) ideal of $K'$ above $\ell$. Writing $k = \mff_{\mfp}(E'/K')$, 
Propositions \ref{prop:iso_red}, \ref{prop:fl_bij}, and \ref{prop:cond_restrict} 
imply that
the fields counted are precisely those with $k \in \{0, 2 \}$ or $k \in \{0, \frac{\ell + 3}{2} \}$ for $\ell \nmid D$ and
$\ell \mid D$ respectively, 
so that the Brauer relation (\ref{brauerreln}) implies that $v_\ell(\Disc(E))\in\{\ell-2+k \}$ with $k$ as above.

If instead $\ell \equiv 3 \pmod 4$ and $\ell \mid D$,
then we have
$
\Disc(K') = \ell^{\ell-3} (-D/\ell)^\frac{\ell-1}{2}
$ and
\begin{equation}
v_\ell(\Disc(E')) \in \{ (\ell-3)\ell+(\ell-1)k:k \in \{0, 2, (\ell + 5)/2 \}\;,
\end{equation}
with $k \neq 2$
because the $\ell$-adic valuation of the discriminant of 
a degree $\ell$ 
field cannot be $\ell - 1$.

For each prime $q\neq\ell$ dividing $D$, $E'/K'$ is unramified at primes over $q$, so that by (\ref{brauerreln}) we have
$v_q(\Disc(E))=v_q(\Disc(K'))$. Also, $E$ must be totally real, because $K'$ is and $[E' : K']$ is odd.
Put together, in all cases
this shows that $\Disc(E)$ is equal to $D^{\frac{\ell-1}{2}}$ times a power of $\ell$ as prescribed in Theorem \ref{thm:main},
finishing the proof.
\end{proof}

\begin{proof}[Proof of Theorem \ref{thm:main_pos}]
The proof is essentially identical, now using the $D > 0$ case of Proposition \ref{prop:size_gb}, applying the identity
$\ell \cdot \frac{\ell^a - 1}{\ell - 1} + 1 = \frac{\ell^{a + 1} - 1}{\ell - 1}$, and obtaining the signature of $E$ by
\eqref{brauerreln}.
\end{proof}

\section{Numerical testing}\label{sec_numerics}
Our work began with $\ell = 5$, by inspecting the Jones-Roberts database of number fields \cite{JR} and finding patterns which called for explanation. 
However, for $\ell > 5$, this database
does not contain enough fields for a reasonable test, and does not include the Galois conditions featuring in our theorems.

We therefore wrote a program using {\tt Pari/GP} \cite{pari} to compute the relevant number fields, for which source code is 
available from the third author's website\footnote{\url{http://www.math.sc.edu/~thornef/}}. A few comments on this program:

Thanks to the relation $\Disc(E)=\Disc(F)\N(\f(E'/F))$ given after
(\ref{eqn:brauer_1}), to enumerate $F_{\ell}$ fields (possibly with
certain conditions, including discriminant and/or Galois restrictions), it
is enough to enumerate suitable $C_{\ell-1}$ fields $F$ (which is very easy),
and for each such field to enumerate suitable conductors $\f$ of 
$C_{\ell}$-extensions $E'/F$ such that $E'/\Q$ is Galois. Luckily, these
Galois conditions imply that these suitable conductors are very restricted,
since they must be of the shape $\f=n\a$, where $n$ is an ordinary integer
and $\a$ is an ideal of $F$ divisible only by prime ideals of $F$
which are above ramified primes of $\Q$, and in addition which must be
Galois stable.

For each conductor $\f$ of this form, we compute the
corresponding ray class group, and if it has cardinality divisible by $\ell$, 
we compute the corresponding abelian extension, and check which subfields of 
degree $\ell$ of that extension satisfy our conditions. 

Note that for our purposes, we only \emph{count} the $F_{\ell}$-extensions
that satisfy our conditions. Our program can also compute them explicitly
thanks to the key {\tt Pari/GP} program \url{rnfkummer}, for which the 
algorithm is described in detail in Chapter 5 of the first author's book 
\cite{Co}.

\smallskip

Our numerical testing was moderately extensive for $\ell = 5$ and $\ell = 7$, 
and rather limited for $\ell = 11$ and $\ell = 13$, as the complexity of our algorithms
grows rapidly with $\ell$. We verified our results and found $F_\ell$-fields with all the
powers of $\ell$ given in our main theorems, with the exception of $13^4$. The computational complexity of our
algorithm severely limited the amount of testing we could conduct with $\ell = 13$; we speculate that the power of $\ell$ not found
is uncommon but does exist.

\bibliographystyle{alpha}
\bibliography{nakagawa}

\end{document}